\documentclass[a4paper,leqno]{amsart}

\pdfoutput=1

\usepackage{amsmath}
\usepackage{amsthm}
\usepackage{amssymb}
\usepackage{amsfonts}
\usepackage{mathrsfs}
\usepackage{color}
\usepackage[pdftex]{graphicx}
\usepackage{nonfloat}

\def\R{{\mathbb R}}
\def\N{{\mathbb N}}
\def\Z{{\mathbb Z}}
\def\T{{\mathbb T}}
\def\E{{\mathbb E}}
\def\P{{\mathbb P}}

\let\oldmarginpar\marginpar
\renewcommand\marginpar[1]{\-\oldmarginpar[\raggedleft\footnotesize #1]%
{\raggedright\footnotesize #1}}

\newcommand{\rst}[1]{\ensuremath{{\raise-.5ex\hbox{\tiny$\arrowvert$}}%
\raise-.5ex\hbox{\tiny$#1$}}}

{\catcode `\@=11 \global\let\AddToReset=\@addtoreset}
\AddToReset{equation}{section}

\theoremstyle{plain}
\newtheorem{theorem}{Theorem}[section]
\newtheorem{lemma}[theorem]{Lemma}
\newtheorem{corollary}[theorem]{Corollary}

\newtheorem{assumption}{Assumption}

\theoremstyle{definition}

\newtheorem{remark}[theorem]{Remark}
\newtheorem*{remark*}{Remark}
\newtheorem*{lemma*}{Lemma}

\begin{document}

\title[Hydrodynamic limit, chemotaxis in   
given heterogeneous environment] 
{A hydrodynamic limit for chemotaxis in a given heterogeneous 
environment \vskip0.5cm
{\tiny {\sl In honor of Willi J\"ager, on the occasion of his 75th birthday.} }}

\author[S.\ Gro{\ss}kinsky]{Stefan Gro{\ss}kinsky}
\address[]{University of Warwick, Mathematics Institute, Zeeman Buildg., 
Coventry CV4 7AL, UK}
\email{S.W.Grosskinsky@warwick.ac.uk}
\author[D.\ Marahrens]{Daniel Marahrens}
\address[]{Max-Planck-Institute for Mathematics in the Sciences  (MPI MIS),
Inselstr.  22, D-04109 Leipzig, Germany}
\email{}
\author[A.\ Stevens]{Angela Stevens}
\address[]{Westf\"alische Wilhelms-Universit\"at M\"unster, 
Applied Mathematics, Einsteinstr. 62, D-48149 M\"unster, Germany}
\email{angela.stevens@wwu.de}

\begin{abstract}
In this paper the first equation within a class of well known
chemotaxis systems is derived as a hydrodynamic limit from a stochastic 
interacting many
particle system on the lattice. The cells are assumed to interact with 
attractive chemical molecules on a finite number of lattice sites, but they only
directly interact among themselves on the same lattice site. The chemical environment is assumed
to be stationary with a slowly varying mean, which results in a non-trivial
macroscopic chemotaxis equation for the cells. Methodologically  
the limiting procedure and its proofs are based on results by Koukkus \cite{Kou99} and 
Kipnis/Landim \cite{KL99}.
Numerical simulations extend and illustrate the theoretical findings.
\end{abstract}

\date{\today}



\maketitle

\tableofcontents

\section{Introduction and result}
\label{sec:intro}

In this paper we derive chemotaxis-like equations 
as a hydrodynamic limit of a stochastic lattice gas. Chemotaxis
describes the directed motion of mainly biological species towards higher or
lower concentrations of chemical signals. Here we consider positive chemotaxis
of \emph{cells},
i.e. motion towards higher concentrations of a chemical
signal, therefore the chemical signal is denoted as \emph{chemo-attractant}.\\


Keller and Segel proposed in
\cite{KS70} a phenomenological chemotaxis model on the macroscopic level
for the aggregation and self-organization of the cellular slime mold amoeba
{\sl Dictyostelium discoideum} (Dd).
This can be written as
%
\begin{align}\label{cells}
 \partial_t \rho &= \nabla\cdot \big(k(\rho, \vartheta)\nabla \rho 
- \chi (\rho,\vartheta)\nabla \vartheta \big)\\
&= \nabla\cdot \big(k(\rho, \vartheta)\nabla \rho - \tilde \chi (\rho,\vartheta)
\rho \nabla \vartheta\big), \nonumber \\
 \partial_t \vartheta &= \Delta \vartheta + r(\rho,\vartheta),\label{CA}
\end{align}
where $\rho$ is the (volume) density of amoebae, $\vartheta$ is the density of the attractive
chemical molecules, $k$ and $\tilde \chi = \chi/\rho$ are functional 
parameters describing the strength of random motion of the cells and 
their chemotactic sensitivity, 
respectively, and $r$ comprises the reaction mechanisms. 
In \cite{KS70}, this cross-diffusion system was motivated by the macroscopic 
phenomenon
being experimentally observed, namely movement of the amoebae in direction of higher concentrations of a chemical signal, i.e. their movement up chemical gradients.
Such general types of chemotaxis systems are relevant 
also
in the context of other biological species. One expects this system to be an accurate description for chemotaxis phenomena occuring in 
systems of many cells and
signal molecules. There is a large literature on the formal derivation of  
macroscopic equations from
microscopic particle systems in such settings. Rigorously this has been 
achieved in much fewer cases, e.g.
by hydrodynamic 
or so-called moderate limits. 
The lack of an ellipticity condition for the limiting PDE-system, respectively
the strong clustering of the cells is a major technical problem in this case.

The first rigorous derivation of a general class of chemotaxis systems from a stochastic
interacting many particle system was given in
\cite{Ste00}, where the cells and the chemical molecules interact moderately
with each other
in the sense described in
\cite{Oel89}. The motion of the cells and the molecules are governed by interacting
stochastic differential equations, and production and decay of the
chemical molecules are modeled via Poisson point processes.
In \cite{Ste00} explicit error estimates could be given, which show that the 
total
number of particles does not have to be too large for the limiting PDE-system to be
a good approximation for the particle model.
The moderate rescaling requires that in the limiting procedure, when the total 
particle number (cells and chemical molecules) tends to
infinity, the main range of interaction of the particles 
tends to zero. 
The number of particles sensed in
this main range of interaction tends to infinity too, while this number is 
still only a vanishing fraction of the total particle number. This is sufficient for 
correlations to become small enough in the limit. The main technical
complication in the 
proof is the cross-diffusion structure of  
the chemotaxis system. The ellipticiy
condition, heavily used in \cite{Oel89}, is no longer valid in this case. 
To overcome this problem, two shadow systems were introdcued, in order
to freeze the critical non-linearity 
and a priori estimates were derived. This technique applies also to
more general systems of PDEs.
\\

In principle it would be desirable to also derive 
chemotaxis like systems as hydrodynamic limits in the sense of 
\cite{KL99}.
In this paper, we rigorously derive equation (\ref{cells}) 
from a microscopic stochastic many particle system on a 
lattice, so only one equation, not the full system.  
The cells are 
assumed to interact via a finite number of lattice sites 
with the chemical molecules.
This corresponds to the assumption that the cells detect chemical
molecules in a fixed small region around themselves. They do not
change the chemical environment in this case. 
The interaction of the cells among themselves takes place only at the 
same lattice site and not among cells located on neighboring sites.
This interaction is described by a function $g$
which among others fulfills conditions comparable to an ellipticity condition 
for the related limiting PDE. These will be specified later.
So we consider the case, where a too strong
clustering/aggregation of cells - a typical and important effect of 
chemotaxis and 
self-organization in Dd - is avoided.

Such an interaction with a finite number of lattice sites introduces
specific mathematical difficulties. One has to deduce a
weak equation for the limit density from {empirical
measures} (see Section \ref{sec:HL}) which converge weakly, cf.\ the 
discussion in \cite[Introduction]{Oel85}.
To be more
precise, we are  
modeling the cells/particles by an interacting particle system as
introduced in \cite{Spi70} and derive limit equations via a hydrodynamic
limit, analogous to the procedure in \cite{KL99}. Interacting particle systems 
in this sense are continuous-time Markov
jump processes which also involve discrete particles moving on a lattice. 
%
%

\medskip

So far we were not able to derive the full system \eqref{cells}, \eqref{CA} as
hydrodynamic limit. It seems to be
a major challenge to identify
suitable particle models with invariant product measures. This would be 
one option in order to prove a hydrodynamic limit for
systems of equations such as \eqref{cells} and \eqref{CA}. Hydrodynamic limits 
without product measures 
for single-species monotone particle systems where established in 
\cite{BGRS06, BGRS13}. Stochastically monotone systems preserve a partial order on state space over time, which allows the use of coupling techniques following first results in \cite{Rez91}. Monotonicity usually leads to homogeneous mass distributions, and indeed it has been shown recently that particle systems with invariant product measures that exhibit a particular form of clustering are necessarily non-monotone \cite{rafferty}. 
%
%
%
%

A generic example which
exhibits invariant product measures is given by the
\emph{zero-range process} (ZRP), where the jump rate of particles at any given
site depends only on the occupation number at this site. It is known, see 
\cite{GS03}, 
that multi-species ZRPs have invariant product measures if the jump rates of
particles satisfy certain symmetry relations, i.e. the \emph{Onsager relations}.
However, in our case this would not allow for purely diffusive motion of the 
chemo-attractant
in \eqref{CA}. Instead, the attractive chemical molecules would need to undergo
a kind of chemotactic motion too, in order for the Onsager relations to 
hold. From the modeling point of view this is definitely not the case
here.

Another method to derive a system of a PDE coupled to an ODE via a 
hydrodynamic limit 
without the use of invariant product
measures was introduced in \cite{DMLP}
%
%
for a specific two-species cellular system
on the lattice.  
The authors new method employs energy estimates for the mesoscopic 
empirical averages of the particles occupation number by which they
obtain $H_1^2$ a-priori bounds and thus are able to derive a substitute for
the two block estimates. For the most crucial terms in their setting
homogenization techniques are used, which play the role of the 
one block estimates.
For a discussion
of the respective particle model itself see also \cite{LT}. 
It would be interesting to see whether in our case one can 
find scaling regimes where the existence of gradients 
(and not only densities as in \cite{DMLP})
can rigorously be proved for the full limiting PDE system.

One typical
approximation of system \eqref{cells}, \eqref{CA} is to
assume a quasisteady chemo-attractant density, i.e.\ to set the left hand side
of \eqref{CA} equal to zero. Formally, this is justified by assuming that the
diffusivity of the chemo-attractant is much larger than the random motion of the 
cells,
see e.g. \cite{JL}.
%
There are also recent results on random walks in dynamic random environments driven by exclusion processes \cite{avena}, which are based on separation of time-scales argument. 
Otherwise there are only very few rigorous results on hydrodynamic limits for two-species systems, see e.g. \cite{FT04,EO10}, since such processes are in general also non-monotone.

%
%

It is obvious from a modeling point of view, how a full chemotaxis 
model for the behavior of particles
on a lattice could look like, compare e.g. \cite{OS}. One ansatz relates
to attractive reinforced random walks for many particles with diffusion 
and decay
of the attractive weight. Deriving a hydrodynamic limit for such, or similar
models would be the final goal. 
We expect the Keller-Segel model to also hold  in this context, but
a proof is still missing.\\

As a first step towards understanding chemotaxis as a hydrodynamic limit, we
therefore consider chemotaxis of cells in a stationary, but random, 
environment. 
In order to obtain a non-trivial macroscopic
chemotactic motion, we impose a slowly varying mean on the 
random chemical environment. 
There have been several studies of
hydrodynamic limits of particles in random media: The hydrodynamic limit for a 
ZRP a with stationary, ergodic environment on the
sites has been obtained in \cite{Kou99}, and corresponding large deviations have been considered in \cite{Kou2}. The case of a stationary, 
ergodic environment on the edges was treated
in \cite{GJ08}. The hydrodynamic limit for the ZRP with slowly varying but
not random 
environment is given in \cite{CR97}. For a result unifying all these situations 
of ``locally convergent'' media in the
special case of the simple exclusion processes compare \cite{Jar10}. Our main 
result in the present paper concerns the modeling of chemotaxis via
a ZRP in a random environment with slowly varying mean
and it is summarized in Theorem \ref{T}, Section \ref{sec:HL}. A special case of
this result is given in Corollary \ref{C} below.\\

At each site $x\in\T^d_N$ on the periodic lattice 
$\T^d_N = \{ 1, \ldots, N \}^d$ one distributes $\zeta(x)\in\N$ molecules
of the 
chemo-attractor such that $\zeta(x)$ is a Poisson random variable 
with parameter $ 
\vartheta(\frac{x}{N})$ and these are kept
fixed for all time. Let $f:\N\to [a,b]$ for some $a,b>0$ and set
\[
 \widetilde{f}\big( 
\vartheta({\textstyle\frac{x}{N}})\big) 
:= \E\bigg[ \frac{1}{f(\zeta(x))} \bigg]^{-1}.
\]
Now we specify the type of chemotactic motion we consider here for the cells.
We start with an initial distribution of $\eta(x)$ cells at $x\in\T^d_N$. All 
cells perform 
independent random walks with a site-dependent jump-rate which is 
$N^2 f(\zeta(x))$. Thus each cell remains at its current
site $x\in\T^d_N$ for an exponentially distributed random waiting time with parameter $N^2 f(\zeta(x))$ and then jumps
to a random neighboring site on the lattice. In this way, particles are more likely to stay at a site $x$ if
$f(\zeta(x))$ is large and therefore $f$ is one way of formulating the  
microscopic 
behavior which results in a chemotactic drift towards higher concentrations
of the chemical signal. To take the
limit as $N\to\infty$, we embed the discrete torus $\T^d_N$ into the continuous torus $\T^d = \R^d/\Z^d$ via $x\mapsto
\frac{x}{N}$. Thus the number of cells in an interval $I \subseteq \T^d$ is given by
\[
 \#\{ \text{cells in $I$} \} = \sum_{x\in N I \cap \T^d_N } \eta(x),
\]
where $N I$ is $I$ rescaled by $N$. 
The cell density in $I$ is then obtained as \\
$\sum_{x\in N I} \eta(x)/(|I| N^d)$.
%
%
\begin{corollary}\label{C}
 Let $\eta_0(x)$ be distributed with a profile $\rho_0(u)$, i.e.\ for each interval $I\subseteq \T^d$, the density
converges in probability:
\[
 \lim_{N\to\infty} \frac{\#\{ \text{cells in $I$} \}}{|I| N^d} = \int_{I} \rho_0(u)\;du.
\]
Then for all later times $t > 0$ and all intervals $I$, almost surely with respect to the distribution of chemical molecules $\zeta$, it holds that
\[
 \lim_{N\to\infty} \frac{\#\{ \text{cells in $I$} \}}{|I| N^d} = \int_{I} \rho_t(u)\;du
\]
in probability, where $\rho_t$ solves
\begin{equation}\label{cells_IRW}
 \partial_t \rho_t(u) = \Delta \big[\widetilde{f}(\vartheta(u))\rho_t(u)\big ] = \nabla\cdot\Big[
\widetilde{f}(\vartheta(u))\nabla\rho_t(u) 
+ \widetilde{f}'(\vartheta(u))\rho_t(u) \nabla\vartheta(u) \Big],
\end{equation}
which is of the form \eqref{cells}.
\end{corollary}
For the derivation of this result from Theorem \ref{T} we refer to the 
discussion in Section~\ref{sec:sim}. Let us first make
several remarks concerning this result.
\begin{remark}
(1) The assumption of independent random walks treated in Corollary \ref{C} yields a
linear limit equation and is much weaker than necessary, see Theorem \ref{T}. It is presented here in order to give the reader an
intuition for the connection between the macroscopic bias 
$\widetilde{f}(\vartheta(u))$ induced by the chemical signal and the 
related microscopic
bias $f(\zeta(x))$.\\
(2) Note that instead of the Poisson distribution, we can let the environment be given by any distribution in $\N$
depending continuously (in distributional topology) on its parameter.\\
(3) In case of a stationary chemical environment, the solution is always 
global 
and no blow-up can occur, thus excluding a
prominent feature of the full chemotaxis system, where the possibility 
of finite time blow-up in
two space dimensions for a suitable parameter range is of
important biological relevance for self-organization phenomena in the
cellular slime mold amoebae Dd, \cite{CP}, \cite{JL}.\\
(4) Finally, let us mention that some microscopic descriptions for chemotactic
cell motion 
are based on the feature that the cells can detect concentration gradients along 
their cell surface.
Our result supports the idea that chemotactic effects can
also occur if the cells only sense the absolute values of the 
concentration of the chemo-attractant without needing to sense its (local)
gradients. An alternative approach would be to let each particle perform asymmetric random walks (and its nonlinear
versions involving a zero range interaction $g(\cdot)$, see
below) with jump rates across edges of the lattice proportional to $p(x,y)$. 
Here $x$, $y$ are neighboring sites, and $p(x,y)$ is a random variable,
which in our
case is determined by the attractive chemical environment $\zeta(x)$., 
see also \cite{OS} and the microscopic approaches in \cite{Alt}.
A typical example for an asymmetric random walk would be 
$p(x,y)=f\big(\zeta(y)-\zeta(x)\big)$, representing a microscopic gradient
for an appropriate function $f$. This approach would lead us to 
consider problems of the type investigated in 
\cite{GJ08}. \\
(5) For positive chemotaxis effects, i.e. clustering/aggregation of the cells
the sign of the second term on the right hand side of (\ref{cells}), (\ref{cells_IRW}) is 
crucial.
Further, already in \cite{Schaaf} it was pointed out that the diffusivity
$\tilde f$, $k$ 
and the chemotactic sensitivity $ - \tilde f'$, $\tilde \chi$ of the cells 
are not completely independent 
functionals. 
A suggestion in \cite{Schaaf} is that the diffusivity $\tilde f$ 
in (\ref{cells_IRW}) relates to the
chemotactic sensitivity $ - \tilde f'$, 
which is supposed to have a positive sign
in case of positive chemotaxis, as follows: 

$ \tilde f' = - \chi_0 \tilde f \tilde\Phi'(\vartheta)$ for a suitable
function $\tilde\Phi$, e.g. $\tilde\Phi(\vartheta) = \vartheta$ and 
$\chi_0 > 0$. 

This results in $\tilde f = C \exp(- \chi_0 \tilde \Phi (\vartheta))$,
e.g. $\tilde f = C \exp ( - \chi_0 \vartheta)$.

So for large chemical
concentrations $\tilde f$ becomes small. 
\end{remark}
The outline of the paper is as follows. In Section \ref{sec:HL}, we present a general result on hydrodynamic limits in a
random environment with slowly varying mean. Its proof is given in 
Section \ref{sec:proof} where we highlight
the differences to the proofs found in \cite{Kou99, KL99}. Then in Section \ref{sec:sim}, we show how to deduce
Corollary \ref{C} from the general result and present numerical simulations of the particle system.

\section{The general result}
\label{sec:HL}
In this paper, we apply the \emph{entropy method} as given in 
\cite{GPV88} to obtain a generalization of equation 
\eqref{cells_IRW}, namely for the chemotactically moving cells, in a stationary, yet random, environment, which represents the density of a
heterogeneously distributed attractive chemical signal. To avoid 
technicalities concerning boundary conditions, we shall work in a
periodic domain. Thus we consider a Markov (Feller) process, describing the
motion of the cells, with state space $\N^{\T^d_N}$, where 
$\T^d_N = \{ 1, \ldots, N \}^d$ denotes the perodic lattice with $N+1 \equiv 1$.
Elements of the state space will be called \emph{particle configurations} and
denoted by the Greek letters $\eta, \zeta, \xi$. Thus $\eta(x) \in \N$ denotes the
number of particles of the configuration $\eta$ at site $x$. The elements of
$\T^d_N$ will be denoted by the letters $x,y,z$ and are called
\emph{microscopic} variables. Two microscopic sites $x, y\in \T^d_N$ are called
\emph{neighbors}, in short $x\sim y$, if $|x-y|=1$. The discrete torus is
embedded in the continuous torus $\T^d = \R^d/\Z^d$ via $x\mapsto x/N$. Elements
of the continuous torus will be denoted by $u$ and are called
\emph{macroscopic} variables. For simplicity, we consider only symmetric nearest-neighbor
jumps; then the Markov process is given by the generator
\begin{equation}\label{L}
 \mathcal{L}_N f(\eta) = \sum_{x\sim y} g(\eta(x)) p^N_x \big(f(\eta^{x,y}) - f(\eta) \big)
\end{equation}
for all $f \in C_b(\N^{\T^d_N})$, where the sum is taken over all (ordered) pairs of neighbors $x$ and $y$. Here
$\eta^{x,y}$ denotes the configuration
obtained from $\eta$ after one particle has jumped from site $x$ to $y$ and $p^N_x > 0$ 
describes the chemical environment. Throughout this article, we assume that 
\[
 p^N_x = v({\textstyle\frac{x}{N}}) + q_x,
\]
where $v \in C^1(\T^d, \R)$ describes the slowly varying mean of the environment and 
$( q_x )_{x\in\Z^d} \subset \R^{\Z^d}$ is a uniformly bounded \emph{stationary} and
\emph{ergodic} sequence of random variables with zero mean. In order to avoid degeneracies, we
assume strictly positive jump rates, i.e.\ without loss of generality we assume $p^N_x, q_x \in [a,b]$ with $0<a<b$.
We denote the law
of $q$ by $m$, so that $m$ is a probability measure on $\R^{\Z^d}$.\\
It is also possible to think of
$p^N_x$ as a random variable. Note, however, that for technical reasons in what follows we fix the ergodic part $q$ of
the environment independently of $N$.
The dependence of the jump rates on the number of cells is given by the function $g : \N \to
[0,\infty)$. In addition to the standard condition
\[
g(n)=0\quad\Leftrightarrow n=0
\]
to avoid degeneracies, we make the following regularity assumptions on $g$ throughout the paper.
\begin{assumption}\label{A}
 (i) Suppose that $g$ is uniformly Lipschitz-continuous, i.e.\ there exists a
constant $g^*$ such that
\[
 \sup_{n\in\N} |g(n+1) - g(n)| \le g^*.
\]
 (ii) Further assume that $g$ grows at least linearly, i.e.\ there
exists $g_0 > 0$ such that
\[
 \inf_{n\in\N} \frac{g(n)}{n} \ge g_0.
\]
\end{assumption}
These assumptions are not optimal but they are standard in the literature, see
e.g.\ \cite[Theorem 5.1.1]{KL99}, on which our proof is based.
Under these assumptions, the process with generator \eqref{L} has invariant
product measures $\nu^{N,p}$ that satisfy the \emph{detailed balance
condition}
\begin{multline*}
 \quad p^N_x g(n) \nu^{N,p}[\eta(x) = n] \nu^{N,p}[\eta(y) = k]\\
 = p^N_y g(k+1) \nu^{N,p}[\eta(x) = n-1] \nu^{N,p}[\eta(y) = k+1]
\end{multline*}
for all neighbors $x,y\in\T^d_N$ and $k,n\in\N$. 
\\
Indeed for any $\varphi \geq 0$, there exists such a measure
$\nu^{N,p}_\varphi$ given by 
\begin{equation}\label{GC}
 \nu^{N,p}_\varphi(\eta) = \prod_{x\in\T^d_N} \frac{\left[ (p^N_x)^{-1}
\varphi \right]^{\eta(x)}}{Z((p^N_x)^{-1} \varphi) g(\eta(x))!} \quad ,
\end{equation}
where $g(n)! = g(1) g(2) \ldots g(n)$, $g(0)! = 1$ and
\begin{equation}\label{Z}
 Z(\varphi) = \sum_{n=0}^\infty \frac{\varphi^n}{g(n)!}
\end{equation}
is the partition function. 
The detailed balance condition ensures that the
generator $\mathcal{L}_N$ in \eqref{L} is 
$L^2(\nu^{N,p}_\varphi)$-selfadjoint and in particular the measure
$\nu^{N,p}_\varphi$ is an invariant product measure. Let
$\nu^1_\varphi$ denote the one-site marginal without environment, i.e.
\[
 \nu^{1}_\varphi(n) =\frac{\varphi^n}{Z(\varphi) g(n)!} \quad .
\]
The associated density then is 
\begin{equation}\label{M_func}
 M(\varphi) = \E_{\nu^1_\varphi} [\eta(0)].
\end{equation}
Assumption \ref{A} implies that $Z$ is finite on $[0,\infty)$. Hence the product
measure $\nu^{N,p}_\varphi$ exists for all $\varphi \in [0,\infty)$ and environments $(p^N_x)_{x\in\T^d_N}$. 
The parameter $\varphi \geq 0$ is called fugacity, and it controls the expected particle density which is given by
\begin{equation}\label{R_func}
 R(u,\varphi) := \E_m\Big[M\Big( {\textstyle\frac{\varphi}{v(u) + q_0}}
 \Big)\Big] 
\end{equation}
for all $u\in\T^d$. The family of distributions (\ref{GC}) is also called the grand-canonical ensemble.
%
%
One can show that for each $u\in\T^d$, the function $R(u,\cdot)$ is strictly
increasing, cf.\ \cite{KL99}, and we then define $\Phi(u,\cdot)$ to be its inverse
function. Note that
\[
 \E_m\Big[\E_{\nu^{N,p}_{\Phi(\frac{x}{N},\rho)}}\big[\eta(0)\big]\Big] = \rho
\quad\text{and}\quad \E_m \Big[\E_{\nu^{N,p}_\varphi } \big[g(\eta(x)) p^N_x \big] \Big] =
\Phi({\textstyle\frac{x}{N}},\varphi).
\]
We shall see that the limit equation is given by
\begin{equation}\label{lim_eqn}
 \partial_t \rho_t(u) = \Delta \Phi(u,\rho_t(u)).
\end{equation}
 Finally we need a suitable notion of distance between probability measures, which in our case is 
given by the relative entropy. Consider measures $\nu, \mu \in P(\N^{\T^d_N})$ such that $\mu$ is absolutely
continuous with respect to $\nu$ (i.e.\ $\mu \ll \nu$). Then the relative entropy is given by
\[
 H(\mu|\nu) = \int_{\N^{\T^d_N}} \log\frac{d\mu}{d\nu}(\eta) \;d\mu(\eta),
\]
where $\frac{d\mu}{d\nu}$ denotes the Radon-Nikodym derivative of $\mu$ with respect to $\nu$.
The main result of this section is the following.
\begin{theorem}\label{T}
 Let $\mu_0^N \in P(\N^{\T^d_N})$ be the initial datum of the particle process
with an associated initial profile $\rho_0 \in L^\infty(\T^d)$, i.e.\ it holds
\[
 \lim_{N\to\infty} \mu_0^N\bigg( \Big|\frac{1}{N^d} \sum_{x\in\T^d_N}
G({\textstyle\frac{x}{N}}) \eta(x) - \int_{\T^d} G(u) \rho_0(u) \;du \Big| \geq
\delta \bigg) = 0 \qquad\text{$m$-almost surely}
\]
for all $G \in C(\T^d)$ and $\delta > 0$. Furthermore we suppose the bounds 
\[
 H(\mu_0^N | \nu^{N,p}_\varphi) \le C N^d \quad\text{and}\quad \E_{\mu_0^N}\bigg[
\sum_{x\in\T^d_N} \eta(x)^2 \bigg] \le C N^d
\]
to hold for some (and hence all) $\varphi > 0$.
Denote by $\mu^N_t \in P(\N^{\T^d_N})$ the measure obtained from the evolution
of the ZRP with rate function $g$ and by $\rho_t \in L^\infty(\T^d)$ the density
obtained from equation \eqref{lim_eqn}. Then under Assumption \ref{A}, it holds
that
\[
 \lim_{N\to\infty} \mu^N_t \bigg( \Big|\frac{1}{N^d} \sum_{x\in\T^d_N}
G({\textstyle\frac{x}{N}}) \eta(x) - \int_{\T^d} G(u) \rho_t(u) \;du \Big| \geq
\delta \bigg) = 0 \qquad\text{$m$-almost surely}
\]
for all $t > 0$, $G \in C(\T^d)$ and $\delta > 0$.
\end{theorem}
\begin{remark}\label{R:CT}
 In order to fully connect Corollary \ref{C} with Theorem \ref{T}, we note that weak convergence of the empirical
measures
implies convergence of the densities over intervals by the Portmanteau theorem \cite[Theorem 2.1]{Bil99}.
This follows from the absolute continuity (with respect to the Lebesgue measure) of the limit measure $\rho_t(u)\;du$.
\end{remark}
To prove this result, we adapt the method in \cite{Kou99} to a random
environment with slowly varying mean. The paper \cite{Kou99} is based on the
entropy method given in \cite{GPV88} and proves a hydrodynamic limit for a zero range
process in a stationary (ergodic) random environment. The case of a zero range
process in a slowly varying deterministic environment was treated in 
\cite{CR97}. Thus, from a purely technical point of view, our result
is a combination of the previous two. Indeed, our method is close to the
methods presented in \cite{Kou99}. Therefore we mainly highlight the 
differences here. The basic idea of the proof lies
in the fact that the slowly varying mean is locally almost constant and 
hence in
small boxes we are essentially in the situation of \cite{Kou99}.

\section{Proof of Theorem \ref{T}}\label{sec:proof}
The first step of the proof is to obtain a priori bounds and tightness of the
particle process. Indeed we can understand the convergence result of Theorem \ref{T} in terms of the empirical measure
\[
 \alpha^N_\eta(du) = \frac{1}{N^d} \sum_{x\in\T^d_N} \eta(x) \delta_{\frac{x}{N}}(du) \in \mathcal{M}_+,
\]
where $\delta_u$ denotes a Dirac delta, $u \in \T^d$, and $\mathcal{M}_+$ is the space of all positive Radon measures on
$\T^d$.
For all $t\geq0$, the configuration $\eta_t$ is a random variable distributed according to $\mu^N_t$. Note that it
suffices to prove Theorem \ref{T} for all $0\le t \le T$ with a fixed, but arbitrary, $T > 0$. Since for any
environment $p$ the ZRP is a jump process on the state space, $(\eta_t)_{t\in[0,T]}$ is a random variable on the path
space $D([0,T];\N^{\T^d})$. Here $D([0,T];\N^{\T^d})$ denotes the space of functions $[0,T]\to\N^{\T^d}$ that are
right-continuous in time with left limits. We denote the distribution of $(\eta_t)_{t\in[0,T]}$ by $\mu^N$. In the same
vein, $(\alpha^N_{\eta_t})_{t\in[0,T]}$ is a random variable on the path space
$D([0,T];\mathcal{M}_+)$; let $Q^{N,p}$ denote its distribution. The entropy dissipation of $\mu^N_t$ can be defined in
terms of the Dirichlet form
\[
 \mathcal{D}(\mu^N_t | \nu^{N,p}_\varphi) := \int_{\N^{\T^d_N}}
\sqrt{\frac{d\mu^N_t}{d\nu^{N,p}_\varphi}} \mathcal{L}_N
\sqrt{\frac{d\mu^N_t}{d\nu^{N,p}_\varphi}} \;d\nu^{N,p}_\varphi.
\]
Throughout this paper we suppose that the assumptions of Theorem \ref{T} are satisfied.
\begin{lemma}\label{L:a-priori}
 It holds that
\[
 H(\mu^N_t | \nu^{N,p}_\varphi) \le C N^d \ \text{for all}\ t\in [0,T]\quad \text{and} \quad \frac{1}{T} \int_0^T \mathcal{D}(\mu^N_t | \nu^{N,p}_\varphi) \;dt \le C N^d.
\]
\end{lemma}
\begin{lemma}
 For any environment $p$, the sequence of probability measures
$(Q^{N,p})_{N\in\N}$ is tight.
\end{lemma}
\begin{lemma}\label{L:abs_cont}
 For any environment $p$, all limit points $Q^{*,p}$ of $(Q^{N,p})_{N\in\N}$ are
supported on $\{ \pi \in D([0,T];\mathcal{M}_+(\T^d)) : \pi_t (du) \ll du \}$, i.e.\ at each time $t$, the limit measures
are absolutely-continuous with respect to the Lebesgue measure.
\end{lemma}
\begin{lemma}\label{L:weak}
 For any environment $p$, if $\pi_t (du) = \rho_t (u) \;du$ is distributed
according to $Q^{*,p}$, then $\rho_t(u)$ is the unique solution of 
\eqref{lim_eqn} in $L^2((0,T)\times\T^d)$.
\end{lemma}
The proofs of Lemma \ref{L:a-priori} -- \ref{L:weak} are similar to the proofs of the corresponding 
%
%
results for the usual ZRP, see Lemma 2.1 -- 2.4 in \cite{Kou99} and Lemma V.1.5 and V.1.6 in \cite{KL99}. 
Once these results have been obtained, Theorem \ref{T} is a straight-forward consequence. The main
difficulty consists in the proof of the so-called replacement lemma, which is the key to
Lemma \ref{L:weak}. In order to state the replacement lemma, we need
some additional notation to describe averages over mesoscopic blocks. Set
\[
 \eta^l(x) = \frac{1}{(2l+1)^d} \sum_{|x-y|\le l} \eta(y)
\]
and
\[
 V_{x,l}^p(\eta) := \bigg| \frac{1}{(2l+1)^d} \sum_{|x-y|\le l} p^N_y g(\eta(y))
-\Phi({\textstyle\frac{x}{N}},\eta^l(x)) \bigg|.
\]
\begin{lemma}[Replacement Lemma]\label{L:RL}
 For every $\delta > 0$, $m$-almost surely it holds that
\[
 \limsup_{\epsilon\to0} \limsup_{N\to\infty} \mu^N \bigg( \int_0^T
\frac{1}{N^d} \sum_{x\in\T^d_N} V_{x,\epsilon N}^{p}(\eta_t) \;dt \geq \delta
\bigg) = 0,
\]
where $\tau_x$ denotes the translation by $x\in\T^d_N$.
\end{lemma}
In the remainder of this section we will mainly prove Lemma \ref{L:RL},
where we sketch the modifications necessary to adapt the proof in \cite{Kou99} 
in order to take
into account the slowly varying average of the environment $v(\frac{x}{N})$. First we
prove the replacement over small boxes of size $l \in \N$.

\subsection{One block estimate}
\begin{lemma}[One Block Estimate]
 It holds that
\[
 \limsup_{l\to\infty} \limsup_{N\to\infty} \E_{\mu^N}\bigg[ \int_0^T
\frac{1}{N^d} \sum_{x\in\T^d_N} V_{x,l}^{p}(\eta_t) \;dt \bigg] = 0
\]
 $m$-almost surely.
\end{lemma}
 Setting 
\[
 f^{N}_t(\eta) = \frac{d\mu^N_t}{d\nu^{N,p}_\varphi}(\eta), \qquad  F^{N}(\eta) =
\int_0^T f^{N}_t(\eta) \;dt,
\]
we see that
\[
 \E_{\mu^N}\bigg[ \int_0^T \frac{1}{N^d} \sum_{x\in\T^d_N} V_{x,l}^{p}(\eta_t) \;dt \bigg] = \frac{1}{N^d}
\sum_{x\in\T^d_N} \int_{\N^{\T^d_N}} V_{x,l}^{p}(\eta) F^{N}(\eta) \;d\nu^{N,p}_\varphi(\eta) \  .
\]
Lemma \ref{L:a-priori}, together with convexity of the entropy and Dirichlet
form, yields the bounds
\begin{equation}\label{ap-bounds}
\begin{split}
 H^{N,p}( F^{N}) := H( F^{N} \;d\nu^{N,p}_\varphi | \;d\nu^{N,p}_\varphi) &\le C N^d,\\
 \mathcal{D}^{N,p}(F^{N}) := \mathcal{D}( F^{N} \;d\nu^{N,p}_\varphi | \;d\nu^{N,p}_\varphi) &\le C N^{d-2}.
\end{split}
\end{equation}
For future reference, we also note the following expression of the Dirichlet form via the density $F^N$. It holds that
\begin{equation}\label{Dirichlet}
 \mathcal{D}^{N,p}(F^{N}) = \sum_{x\sim y} \int_{\N^{\T^d_N}} \frac{1}{2} p^N_x g(\eta(x)) \big[ \sqrt{F^N(\eta^{x,y})}
- \sqrt{F^N(\eta)} \big]^2
\;d\nu^{N,p}_\varphi \ .
\end{equation}
The next lemma allows us to restrict ourselves to bounded particle
configurations. Its proof is analogous to the proof of \cite[Lemma V.4.2]{KL99} with the entropy inequality
as additional ingredient.
\begin{lemma}
Under the conditions of Theorem \ref{T}, we have
\[
 \limsup_{A\to\infty} \limsup_{l\to\infty} \limsup_{N\to\infty} \sup_{p, f}
\int_{\N^{\T^d_N}} \frac{1}{N^d} \sum_{x\in\T^d_N} V_{x,l}^{p}(\eta) \chi_{\{\eta^l (x) > A\}} (\eta ) f(\eta)
\;d\nu^{N,p}_\varphi(\eta) = 0.
\]
where the supremum is taken over all densities $f$ such that $H^{N,p}(F^{N})
\le C N^d$, $\mathcal{D}^{N,p}(F^{N}) \le C N^{d-2}$ and all environments $p=p^N \subset [a,b]^{\T^d_N}$. For every set $\mathcal{A}$, we denote by $\chi_{\mathcal{A}}$ the characteristic function of $\mathcal{A}$. 
\end{lemma}
Therefore we just need to prove
\[
 \limsup_{l\to\infty} \limsup_{N\to\infty} \frac{1}{N^d} \sum_{x\in\N^{\T^d_N}} \int_{\N^{\T^d_N}} V_{x,l,A}^{p}(\eta)
F^{N}(\eta) \;d\nu^{N,p}_\varphi(\eta) = 0
\]
$m$-almost surely, where we set $V_{x,l,A}^p(\eta) := V_{x,l}^p(\eta) \chi_{\{\eta^l(x) \le A\}} (\eta )$. 
In fact, due to the bounds \eqref{ap-bounds} on the Dirichlet form 
of $F^{N}$, it suffices to find an upper bound for
\begin{equation}\label{limsup_proof1}
 \frac{1}{N^d}\sum_{x\in\T^d_N} \int_{\N^{\T^d_N}} V_{x,l,A}^{p}(\eta) F^{N}(\eta)
\;d\nu^{N,p}_\varphi(\eta) - \gamma C N^{2-d} \mathcal{D}^{N,p}(F^{N})
\end{equation}
for all $\gamma > 0$. Next we restrict the problem to translations of the small box $\Lambda_l := \{ -2l, \ldots, 2l
\}$. Denote by $\nu^{x,l,p}_{\varphi}$ the $\Lambda_{x,l}$-marginal of $\nu^{N,p}_\varphi$ and by $F_{x,l}(\eta)$ the
density of the $\Lambda_{x,l}$-marginal of the measure $F^{N}(\eta) d\nu_\varphi^{N,p}(\eta)$ with respect to
$\nu^{x,l,p}_{\varphi}$. We will sometimes drop the index for $x=0$ for such 
quantities when taken at the origin.
Furthermore we define a Dirichlet form on $\Lambda_{x,l}$ by
\begin{align*} 
 \mathcal{D}^p_{x,l}(f) &= \sum_{\stackrel{y\sim z}{y,z\in\Lambda_{x,l}}} I_{y,z}^p(f),\qquad \text{where we have set}\\
 I_{y,z}^p(f) &= \frac{1}{2}
\int_{\N^{\T^d_N}} p^N_y g(\eta(y)) \Big( \sqrt{f(\eta^{y,z})} - \sqrt{f(\eta)} \Big)^2 d\nu^{N,p}_\varphi.
\end{align*}
By convexity of the ''bond''-Dirichlet forms $I^p_{x,y}$, we have $I^{p}_{y,z}(F_{x,l}) \le I^{p}_{y,z}(F^{N})$ for all
neighbors $x,y\in\Lambda_{x,l}$ and thus
\[
 \frac{1}{N^d}\sum_{x\in\T^d_N} \mathcal{D}^{p}_{x,l}(F_{x,l}) \le \frac{C(l)}{N^d} \mathcal{D}^{N,p}(F^{N}).
\]
Hence instead of expression~\eqref{limsup_proof1}, we just need to 
find an upper bound for
\begin{equation}\label{limsup_proof2}
 \frac{1}{N^d}\sum_{x\in\T^d_N} \bigg\{ \int_{\N^{\T^d_N}} V_{x,l,A}^{p}(\eta)
F_{x,l}(\eta) \;d\nu^{x,l,p}_{\varphi}(\eta) - \gamma C(l) N^{2} \mathcal{D}^{p}_{x,l}(F_{x,l}) \bigg\}.
\end{equation}
Next we need to take care of the random part of the environment. The method to keep track of the environment is due to
\cite{Kou99} and constitutes the main deviation from the proof of the hydrodynamic limit for the usual ZRP as
exhibited in \cite{KL99}. Thus we fix $\alpha, \delta > 0$ and let $n \in \mathbb{N}$ be sufficiently large such that
$\frac{1}{n} < \delta$. Divide the interval $[a,b]$ into sub-intervals of length not greater than $\delta (b - a)$ via
$I^\delta_j = [\beta_j,
\beta_{j+1})$ for $0 \leq j \leq n-2$ where
\[
\beta_j = a + (b - a) \frac{j}{n} \qquad (j = 0, \ldots, n-1)
\]
and $I^\delta_{n-1} = [\beta_{n-1}, b]$. Fix $k < l$ and let $L = [(2l + 1)/(2k + 1)]^d$, where $[x]$ denotes the
Gaussian bracket, i.e.\ the largest integer smaller than or equal to $x\in\mathbb{R}$. Now we divide $\Lambda_l$ into
disjoint cubes of the form $x+\Lambda_k$, where we take $B_i$, $i = 1,\ldots,L-1$, such that
\[
B_i \subseteq \Lambda_l, \quad B_i \cap B_j = \emptyset \text{ for } i \neq j, \quad \text{ and } B_i = x_i +\Lambda_{l}
\text{ for some } x_i \in \mathbb{Z}^d.
\]
Then set $B_L = \Lambda_l \setminus \bigcup_{i=1}^{L-1} B_i$ and without restriction take $x_1 = 0$, i.e.\ $B_1 =
\Lambda_k$. Also set $B_i(x) = x + B_i$ for $x \in \mathbb{Z}^d$.
For $x \in \mathbb{T}^d_N$, $\alpha\in (0,1)$ and $q \in [a,b]^{\Lambda_l}$, set
\begin{equation}\label{Anelka}
\begin{aligned}
N^{l,k,\delta}_{x,j,i}(q) &:= \frac{1}{(2k + 1)^d} \sum_{z \in B_i(x)} \chi_{I^\delta_j }(q_z),\\
A^{l,k,\delta}_{x,i,\alpha} &:= \left\{(q_z)_{z \in \Lambda_l} : \left|N^{l,k,\delta}_{x,j,i}(q) -
\E_m [
\chi_{I^\delta_j}  (q_z) 
]\right| \leq \alpha, j = 0, \ldots, n-1 \right\},\\
A^{l,k,\delta}_{x,\alpha} &:= \left\{(q_z)_{z \in \Lambda_l} : \frac{1}{L} \sum_{i=1}^L \chi_{A^{l,k,\delta}_{x,i,\alpha} }(q) \geq 1 - \alpha \right\}.
\end{aligned}
\end{equation}
Since
$V^{p}_{l,A}\le C(A)$ is bounded, the expression~\eqref{limsup_proof2} is 
bounded form above by
\begin{multline}\label{limsup_proof3}
\frac{1}{N^d}\sum_{x\in\T^d_N} \Bigg\{ \chi_{A^{l,k,\delta}_{x,\alpha}} (p)
\bigg(\int_{\N^{\T^d_N}} V_{x,l,A}^{p}(\eta) F_{x,l}(\eta) \;d\nu^{x,l,p}_{\varphi}(\eta) - \gamma C(l) N^{2}
\mathcal{D}^{p}_{x,l}(F_{x,l}) \bigg )\\
+ C(A) (1-\chi_{A^{l,k,\delta}_{x,\alpha}} (q))\Bigg\}.
\end{multline}
Now we take care of the non-random part of the enviroment. To this end, set
\[
 {\widetilde A}^{l,k,\delta}_{x,u,\alpha} = \Big\{ (p_z)_{z \in \Lambda_l}   : p_z = q_z + v_z, q\in
A^{l,k,\delta}_{x,\alpha}, \sup_{z\in\Lambda_l} |v_{x+z} - v(u)| \le \alpha \Big\}
\]
 Since $v$ is smooth, in particular (uniformly) continuous, for every $\alpha$ we can choose $N$ large enough such
that $v_{x+z}:= v((x+z)/N)$ does not differ from $v(x/N)$ by more than $\alpha$. Thus~\eqref{limsup_proof3} is
bounded from above by
\begin{multline}\label{limsup_proof4}
 \frac{1}{N^d}\sum_{x\in\T^d_N} \Bigg\{ \chi_{{\widetilde A}^{l,k,\delta}_{x,\frac{x}{N},\alpha}} (p)
\bigg(\int_{\N^{\T^d_N}} V_{x,l,A}^{p}(\eta) F_{x,l}(\eta) \;d\nu^{x,l,p}_{\varphi}(\eta) - \gamma C(l) N^{2}
\mathcal{D}^{p}_{x,l}(F_{x,l})\bigg)\\
 + C(A) (1-\chi_{A^{l,k,\delta}_{x,\alpha}} (q))\Bigg\}.
\end{multline}
Since the random part $q$ of the environment is stationary and ergodic by assumption, it holds
\[
\frac{1}{N^d}\sum_{x\in\T^d_N} \Big(1-\chi_{A^{l,k,\delta}_{x,\alpha} } (q)\Big) \xrightarrow{N\to\infty} \P_m(q\notin
A^{l,k,\delta}_{0,\alpha}),
\]
which in turn vanishes by ergodicity in the limit as $l\to\infty$ and $k\to\infty$.
The remaining term in~\eqref{limsup_proof4} is bounded from above by
\begin{multline*}
 \sup_{u\in\T^d}\chi_{ {\widetilde A}^{l,k,\delta}_{[uN],u,\alpha}} (p)
\bigg(\int_{\N^{\T^d_N}} V_{{[uN]},l,A}^{p}(\eta) F_{[uN],l}(\eta) \;d\nu^{[uN],l,p}_{\varphi}(\eta)\\ - \gamma
C(l) N^{2} \mathcal{D}^{p}_{[uN],l}(F_{[uN],l})\bigg).
\end{multline*}
Again $\tau_{-[uN]} p$ is in ${\widetilde A}^{l,k,\delta}_{0,u,\alpha}$ for large enough $N$ and the term involving
$\Phi(\frac{[uN]}{N},\rho)$ which appears in $V^p_{[uN],l,A}$ converges uniformly in $u$ and $\rho\le A$ to
$\Phi(u,\rho)$. Thus we see that, up to an error that vanishes as $N\to\infty$, it suffices to estimate
\begin{equation}\label{limsup_proof5}
 \sup_{u\in\T^d}\sup_{p \in {\widetilde A}^{l,k,\delta}_{0,u,\alpha}}\sup_{f} \bigg(\int_{\N^{\Lambda_l}}
V_{l,A}^{p}(\eta) f(\eta) \;d\nu^{l,p}_{\varphi}(\eta) - \gamma C(l) N^{2} \mathcal{D}^{p}_{0,l}(f)\bigg),
\end{equation}
where the inner supremum is taken over all densities $f$ on $\N^{\Lambda_l}$. For fixed $l\in\N$, the term
$V^p_{0,l,A}$ is only non-zero on the compact space of configurations in $\N^{\Lambda_l}$ of at most $(2l+1)^d A$
particles, the Dirichlet form is lower-semicontinuous and this property is conserved by the supremum. Hence the limit
superior of~\eqref{limsup_proof5} is bounded from above 
by
\begin{equation}\label{limsup_proof6}
 \sup_{u\in\T^d}\sup_{p \in {\widetilde A}^{l,k,\delta}_{0,u,\alpha}}\sup_{f} \int_{\N^{\Lambda_l}}
V_{l,A}^{p}(\eta) f(\eta) \;d\nu^{l,p}_{\varphi}(\eta),
\end{equation}
where now the inner supremum is taken over all densities $f$ on $\N^{\Lambda_l}$ with vanishing Dirichlet form
$\mathcal{D}^{p}_{l}(f) = 0$. A density $f$ has vanishing Dirichlet form if and only if it is constant along all
hyperplanes of a given number of particles. With this in mind, we introduce the \emph{canonical measures}
\[
 \nu^p_{l,K}(\,\cdot\,) := \nu^{0,l,p}_{\varphi}(\,\cdot\, | \, {\textstyle\sum_{x\in\Lambda_l}\eta(x) = K}) \ .
\]
Since the probability density $f$ is constant on $\{\sum_{x\in\Lambda_l}\eta(x) = K\}$, we can
estimate~\eqref{limsup_proof6} from above by
\begin{multline}\label{limsup_proof7}
 \sup_{u\in\T^d_N}\sup_{p \in {\widetilde A}^{l,k,\delta}_{0,u,\alpha}}\max_{0\le K\le (2l+1)^d A} \Bigg\{
\int_{\N^{\T^d_N}} \frac{1}{L}\sum_{i=1}^L \int \Big| \frac{1}{(2k+1)^d} \sum_{x\in B_i} p_x g(\eta(x))\\ -
\Phi\big(u,{\textstyle\frac{K}{(2l+1)^d}}\big) \Big| \;d\nu^{p}_{l,K}(\eta) \Bigg\},
\end{multline}
where we have inserted the explicit form of $V^p_{0,l,A}$ and applied the triangle inequality. For given $u$ and
$p\in{\widetilde A}^{l,k,\delta}_{0,u,\alpha}$, we define $\varphi^p_K$ by
\begin{equation}\label{varphi_p_K}
 \E_{\nu^{l,p}_{\varphi^p_K}} [\eta^l(x)] = \frac{K}{(2l+1)^d}.
\end{equation}
Note that $0\le\varphi^p_K\le C(A)$ is uniformly bounded in $0\le K \le (2l+1)^d A$ and $l\in\N$. The next lemma
concerns the closeness of the grand-canonical and the canonical measures.
\begin{lemma}[Equivalence of ensembles]\label{L:EoE}
 For all $F:\N^{\Lambda_k} \to \R$ with finite second moments with respect to $\nu^{k,p}_{\varphi}$ for all $\varphi \le
C(A)$, it holds that
\[
 \lim_{l\to\infty} \sup_{p\in[a,b]^{\Z^d}} \max_{0\le K \le (2l+1)^d A} \bigg|
\E_{\nu^{l,p}_{\varphi^p_K}} \big[ F(\eta) \big] - \E_{\nu^p_{l,K}} \big[ F(\eta) \big] \bigg| = 0.
\]
\end{lemma}
\begin{proof}
 Thanks to the uniform bounds $a\leq p_x \le b$, the equivalence of ensembles can be proved as explained in
\cite[Appendix 2]{KL99}, see also \cite[Lemma 3.3]{Kou99}.
\end{proof}
 By applying Lemma \ref{L:EoE} to \[F(\eta) = \frac{1}{(2k+1)^d} \sum_{x\in B_i} p_x g(\eta(x)),\] we just need to find an upper bound for
\begin{multline}\label{limsup_proof8}
 \sup_{u\in\T^d_N}\sup_{p \in {\widetilde A}^{l,k,\delta}_{0,u,\alpha}}\max_{0\le K\le (2l+1)^d A} \Bigg\{
\int_{\N^{\T^d_N}} \frac{1}{L}\sum_{i=1}^L \int \Big| \frac{1}{(2k+1)^d} \sum_{x\in B_i} p_x g(\eta(x))\\ -
\Phi\big(u,{\textstyle\frac{K}{(2l+1)^d}}\big) \Big| \;d\nu^{l,p}_{\varphi^p_K}(\eta) \Bigg\}.
\end{multline}
The next two lemmas yield convergence of this expression in the limit as $l\to\infty$ and then $k\to \infty$. In a
certain sense, the first lemma describes the ``ergodic'' error between spatial averages and the quenched grand-canonical average. 
\begin{lemma}\label{lemma_OBE1}
For all $\alpha,\delta > 0$, it holds that
\begin{multline*}
 \limsup_{k\to\infty} \limsup_{l\to\infty} \sup_{u\in\T^d_N}\sup_{p \in {\widetilde
A}^{l,k,\delta}_{0,u,\alpha}}\max_{0\le K\le (2l+1)^d A}\\ \Bigg\{
\int_{\N^{\T^d_N}} \frac{1}{L}\sum_{i=1}^L \int \Big| \frac{1}{(2k+1)^d} \sum_{x\in B_i} p_x g(\eta(x)) -
\varphi^p_K \Big| \;d\nu^{l,p}_{\varphi^p_K}(\eta) \Bigg\} = 0.
\end{multline*}
\end{lemma}
The second lemma describes the difference between the quenched and annealed grand-canonical averages.
\begin{lemma}\label{lemma_OBE2}
 It holds that
\begin{multline*}
 \limsup_{\delta\to 0} \limsup_{\alpha\to 0} \limsup_{k\to\infty} \limsup_{l\to\infty} \sup_{u\in\T^d_N}\sup_{p \in
{\widetilde
A}^{l,k,\delta}_{0,u,\alpha}}\max_{0\le K\le (2l+1)^d A} \Big| \varphi^p_K -
\Phi\big(u,{\textstyle\frac{K}{(2l+1)^d}}\big) \Big|\\ = 0.
\end{multline*}
\end{lemma}
\begin{proof}[Proof of Lemma~\ref{lemma_OBE1}]
 Under the law $\nu^{l,p}_{\varphi^p_K}$, the random variables $(p_x g(\eta(x)) - \varphi^p_K)_{x\in B_i}$ are
independent with zero mean. Furthermore their variance is uniformly bounded by continuity in $p$ and $K$. Thus
the expression to be bounded vanishes by a law of large numbers.
\end{proof}
\begin{proof}[Proof of Lemma~\ref{lemma_OBE2}]
 The absolute value in the given equation is obviously continuous in $p$ and $K$. Using the fact that
$\widetilde{A}^{l,k,\delta}_{0,u,\alpha}$
is closed and $v(\cdot)$ is continuous, it is straight-forward to prove that the supremum over $p$ and $K$ is continuous
with respect to $u$. Hence for every $l\in\N$, the supremum in the hypothesis is achieved for some $u\in\T^d$,
$p\in\widetilde{A}^{l,k,\delta}_{0,u,\alpha}$ and $0\le K \le (2l+1)^d A$. For these maximizers, let $\varphi$ be a
limit point as $l\to\infty$ and $\varphi^p_K$ the corresponding subsequence. 
Then the absolute value in the given formula in Lemma
\ref{lemma_OBE2} is bounded from above by
\begin{multline}\label{lem_OBE2}
  \Big| \varphi^p_K - \varphi\Big| + \Big| \varphi - \Phi\Big(u,{\textstyle
\frac{1}{(2l+1)^d} \sum_{z\in\Lambda_l} M\big(\frac{\varphi}{p_z}\big)}\Big) \Big|\\ + \Big| \Phi\Big(u,{\textstyle
\frac{1}{(2l+1)^d} \sum_{z\in\Lambda_l} M\big(\frac{\varphi}{p_z}\big)}\Big) -
\Phi\Big(u,{\textstyle \frac{1}{(2l+1)^d} \sum_{z\in\Lambda_l} M\big(\frac{\varphi^p_K}{p_z}\big)}\Big) \Big|,
\end{multline}
where we have used \eqref{M_func}, \eqref{GC} and \eqref{varphi_p_K}. By continuity and since $\varphi^p_K$ converges
to $\varphi$ as $l\to\infty$, the first and last terms vanish in the limit. Now by definition of $\Phi$
and~\eqref{R_func}, we write
\[
 \varphi = \Phi(u,R(u,\varphi)) = \Phi\Big(u,\E_m\Big[M\big({\textstyle\frac{\varphi}{v(u) + q_0}}\big)\Big]\Big).
\]
Since $\Phi(u,\cdot)$ is Lipschitz continuous uniformly in 
$u\in\T^d$, the second term in \eqref{lem_OBE2} is bounded 
from above by
a multiple of
\[
  \bigg| \E_m\Big[M\big({\textstyle\frac{\varphi}{v(u) + q_0}}\big)\Big] - { \frac{1}{(2l+1)^d} \sum_{z\in\Lambda_l}
M\big({\textstyle\frac{\varphi}{p_z}}\big)} \bigg|.
\]
By the triangle inequality and since 
$p\in \widetilde{A}^{l,k,\delta}_{0,u,\alpha}$ is of the form $p_z = q_z + v_z$
with $|v_z - v(u)| \le
\alpha$, this term is bounded from above by
\[
 \frac{1}{L}\sum_{i=1}^L \chi_{A^{l,k,\delta}_{0,i,\alpha}} (q)\bigg| \E_m\Big[M\big({\textstyle\frac{\varphi}{v(u)
+
q_0}}\big)\Big] - { \frac{1}{(2k+1)^d} \sum_{z\in B_i} M\big({\textstyle\frac{\varphi}{v(u)+q_z}}\big)} \bigg| +
C\alpha
\]
for all small enough $\alpha > 0$. Hence it is enough to consider
\[
 \sup_{q\in A^{l,k,\delta}_{0,1,\alpha}} \bigg| \E_m\Big[M\big({\textstyle\frac{\varphi}{v(u) + q_0}}\big)\Big] - {
\frac{1}{(2k+1)^d}
\sum_{z\in \Lambda_k} M\big({\textstyle\frac{\varphi}{v(u) + q_z}}\big)} \bigg|.
\]
For each $q$, the absolute value in this formula is bounded from above by
\begin{multline*}
 \bigg| \E_m\Big[M\big({\textstyle\frac{\varphi}{v(u) + q_0}}\big)\Big] - \sum_{j=1}^{n-1}
M\big({\textstyle\frac{\varphi}{v(u)+\beta_j}}\big) m(I^\delta_j)
\bigg|\\
+ \bigg| \sum_{j=1}^{n-1} M\big({\textstyle\frac{\varphi}{v(u)+\beta_j}}\big) \big( m(I^\delta_j) -
N^{l,k,\delta}_{j,1}(q) \big) \bigg|\\
+ \bigg| \frac{1}{(2k+1)^d} \sum_{z\in \Lambda_k} \sum_{j=0}^{n-1} \Big(
M\big({\textstyle\frac{\varphi}{v(u)+\beta_j}}\big) - M\big({\textstyle\frac{\varphi}{v(u)+q_z}}\big) \Big) \chi_{ I^\delta_j } (q_z ) \bigg|.
\end{multline*}
By the Lipschitz-continuity of $M$ and since 
$n \le C\delta^{-1}$, this expression is bounded from above by 
$C(\delta + \alpha \delta^{-1})$
for all $q\in A^{l,k,\delta}_{0,1,\alpha}$. Taking the limit, first for $\alpha \to 0$ and then for
$\delta \to 0$,  finishes the proof of the one block estimate.
\end{proof}

\begin{remark}
The above estimates correspond to the approximation of the integral with respect to the
measure $m$ by an integral of simple functions.
\end{remark}
\vskip0.1cm
\subsection{Two blocks estimate}

Similarly, we can prove the two blocks estimate, which shows that the
difference between averages over boxes of size $l$ and $\epsilon N$ is
negligible in the limit.
\begin{lemma}[Two Blocks Estimate]
 It holds that
\[
 \limsup_{l\to\infty} \limsup_{\epsilon\to0} \limsup_{N\to\infty} \sup_{|y| \le
\epsilon N} \E_{\mu^N}\bigg[ \int_0^T \frac{1}{N^d} \sum_{x\in\T^d_N} \big|
\eta_t^l(x+y) - \eta_t^{\epsilon N}(x) \big| \;dt \bigg] = 0
\]
 $m$-almost surely.
\end{lemma}
We will skip the proof here and just note that it follows 
in a similar manner like the one block estimate --- the main difference
lying in restricting the process to two boxes of size $l$ (that are at most $\epsilon N$ apart) instead of just one.
For details, the reader may refer to \cite{KL99, Kou99}.

%

\section{Outlook and Examples: numerical simulations}\label{sec:sim}



In this section, we present examples for chemotactic motion of particles in a 
slowly varying chemical attractive environment in a slightly more
general setting than our theory has so far addressed. This is done in order
to get further insight into the behavior of such particle systems. Consider the (smooth)
function $\vartheta:\T^d \to (0,\infty)$. Assume that the 
attractive chemical
environment $\zeta(x)$ on the lattice $\T^d_N$ is given by independent Poisson random variables with parameter
$\vartheta(\frac{x}{N})$ at site $x\in\T^d_N$. As before, let $m$ denote the law of this random environment. One
might think of this distribution as resulting from independent diffusion of 
attractive chemical molecules on a time-scale much faster 
than the typical time-scale of the diffusive motion of cells so that it is effectively time independent with respect to the random motion of the cells.
This assumption on the time-scales is biologically reasonable and helpful also for theoretical reasons as outlined in the introduction. 
%
%
We suppose that the cells motion is described by a Markov process with
generator given by \eqref{L}, where the
environment $p^N$ depends on the number of attractive chemical molecules $\zeta$ via
\[
 p^N_x = f(\zeta(x)) \qquad \text{for all $x\in\T^d_N$ .}
\]
Thus $f:\N\to(0,\infty)$ provides a microscopic description of the chemical bias.
For simplicity, let us suppose for the rate $g(\eta) = \eta$, i.e.\ that the cells diffuse independently in the random environment $p^N$.
Note that the ``random part'' $p^N_x - \E_m[f(\zeta(x))]$ consists of independent but not stationary random variables,
so our theory given before does not apply directly in this case. However, on local boxes the density $\vartheta(\frac{x}{N})$ is
approximately constant, say $\vartheta(u)$, and the environment consists of almost identically and
independently distributed particles, so that we expect  that with techniques as given above a hydrodynamic limit can be derived in
this case, too. To identify this limit, consider $M$ given as in \eqref{M_func}. Since $g$ is the identity,
$\nu^1_\varphi$ is a product of identical Poisson distributions with parameter $\varphi$ and it follows $M(\varphi) =
\varphi$. Therefore we see that $R$ as defined in \eqref{R_func} is given by
\[
 R(u,\varphi) = \E_{\text{Pois}(\vartheta(u))}\bigg[\frac{\varphi}{f(\zeta(x))}\bigg],
\]
where the average on the right hand side is taken over a Poisson distribution with parameter $\vartheta(u)$.
Solving for $\varphi$ yields
\[
 \Phi(u,\rho) = \E_{\text{Pois}(\vartheta(u))}\bigg[\frac{1}{f(\zeta(x))}\bigg]^{-1} \rho.
\]
In other words, the effective influence of the environment is given by its \emph{harmonic mean}. This is reminiscent of
analogous formulas in (one-dimensional) stochastic homogenization, cf.\ \cite{PV81}. Let us denote
\[
 \widetilde{f}(\vartheta(u)) := \E_{\text{Pois}(\vartheta(u))}\bigg[\frac{1}{f(\zeta(x))}\bigg]^{-1},
\]
which yields the relationship between the microscopic chemical bias $f$ and its macroscopic counterpart
$\widetilde f$. Analogous to (\ref{lim_eqn}) the equation describing the limit of the empirical measures is then explicitly given by
\begin{equation}\label{simple_lim}
 \partial_t \rho(t,u) = \Delta\big[\widetilde{f}(\vartheta) \rho \big](t,u),\quad (t,u) \in [0,T]\times\T^d,
\end{equation}
which we rewrite as
\[
 \partial_t \rho = \nabla \cdot \Big([\widetilde{f}
(\vartheta)\nabla\rho +
\widetilde{f}'(\vartheta) \rho \nabla \vartheta \Big],
\]
which is of the form \eqref{cells} with $k(\rho, \vartheta) = \widetilde{f}(\vartheta)$ and $\chi (\rho,
\vartheta) = - \rho\widetilde{f}'(\vartheta)$, respectively
$\tilde \chi (\rho, \vartheta) = - \widetilde{f}'(\vartheta)$. From
\eqref{simple_lim} it follows that the
stationary states are given by multiples of
\begin{equation}\label{stead_states}
 \rho_\infty(u) = \frac{1}{\widetilde{f}(\vartheta(u))} =
\E_{\text{Pois}(\vartheta(u))}\bigg[\frac{1}{f(\zeta(x))}\bigg].
\end{equation}
%
%
The respective coefficient is determined by the initial mass, which is 
conserved.
In order to perform numerical simulations, let us suppose that the 
dependence of $f(\zeta)$ on the attractive chemical molecules  
is explicitly given by 
\[
 f(\zeta) = \nu+\frac{\chi_0}{1+\zeta(x)}
\]
with parameters $\nu, \chi_0 > 0$ describing the relative strengths of the free diffusion and the chemotactic motion.
Specifically, we set $\chi_0 = 2$ and $\nu = 0.5$. Furthermore, let
\begin{equation}\label{CA_dens}
 \vartheta(u) = 30\exp\big(-60((u_1 - 0.5)^2 + (u_2 - 0.5)^2)\big).
\end{equation}
The simulations of the stochastic particle system are carried out with an 
algorithm based on random sequential updates. Here the particle densities are
measured over balls in $\T^2$ of radius $0.05$. The results are then compared 
with the numerical solution of the corresponding
limiting PDE, which is obtained by using a standard finite difference method. 
Since we only deal with the first equation of a chemotaxis system and because 
we are not in the situation where blowup in finite time is expected, we
do not go into further details concerning more refined algorithms 
for the PDE in this illustrating section.

The simulations for a lattice size of $N\times N$ with
$N=250$ and uniform initial configuration $\eta(x) = 4$ for all $x\in\T^2_N$ are shown in
Figs.~\ref{fig:ce1}--\ref{fig:PDE5}.
\\[\intextsep]
\begin{minipage}{0.5\linewidth}
\centering
\includegraphics[width=\linewidth]{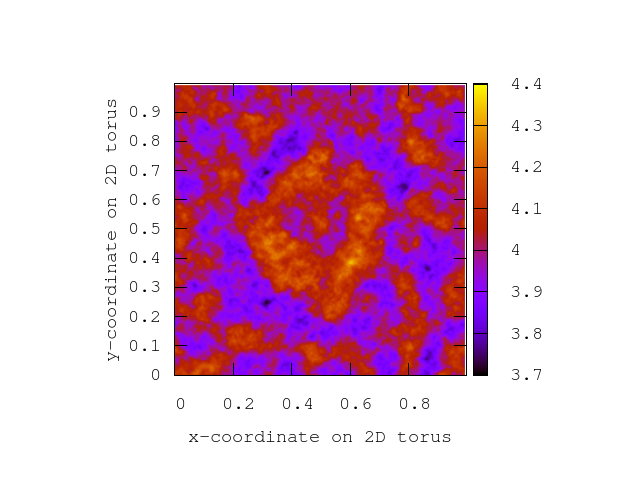}
\figcaption{Cell density for particle system, $t=0.0008$}
\label{fig:ce1}
\end{minipage}
\begin{minipage}{0.5\linewidth}
\centering
\includegraphics[width=\linewidth]{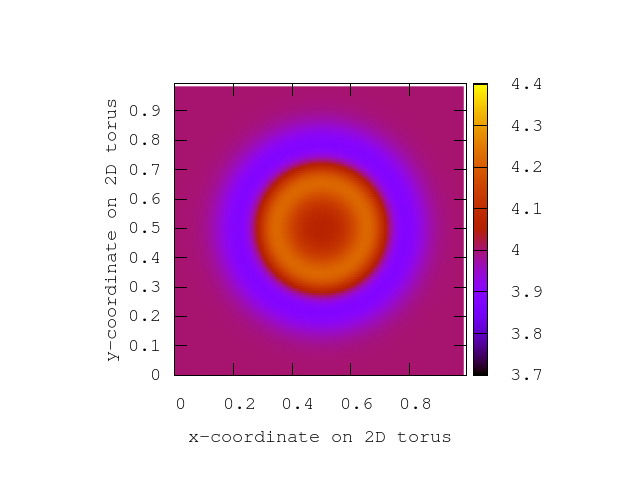}
\figcaption{PDE solution, $t=0.0008$}
\label{fig:PDE1}
\end{minipage}
\\[\intextsep]
\begin{minipage}{0.5\linewidth}
\centering
\includegraphics[width=\linewidth]{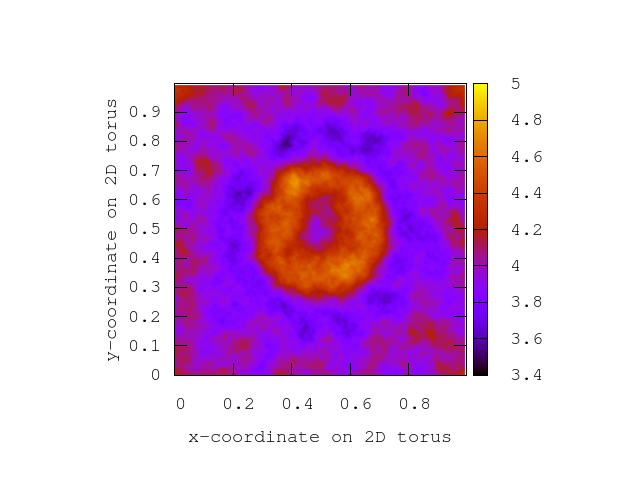}
\figcaption{Cell density for particle system, $t=0.004$}
\label{fig:ce2}
\end{minipage}
\begin{minipage}{0.5\linewidth}
\centering
\includegraphics[width=\linewidth]{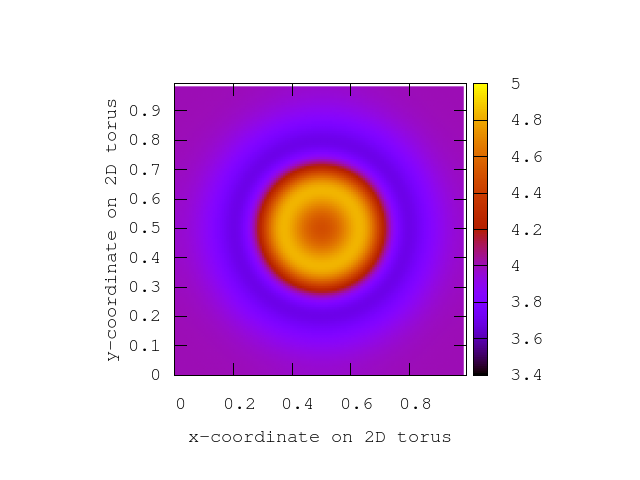}
\figcaption{PDE solution, $t=0.004$}
\label{fig:PDE2}
\end{minipage}
\\[\intextsep]
\begin{minipage}{0.5\linewidth}
\centering
\includegraphics[width=\linewidth]{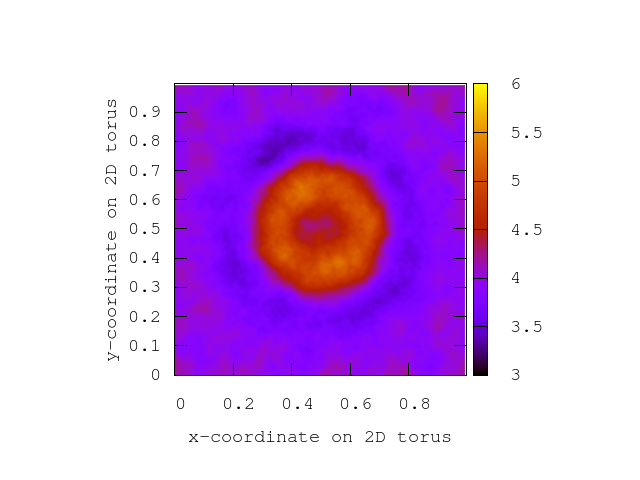}
\figcaption{Cell density for particle system, $t=0.01$}
\label{fig:ce3}
\end{minipage}
\begin{minipage}{0.5\linewidth}
\centering
\includegraphics[width=\linewidth]{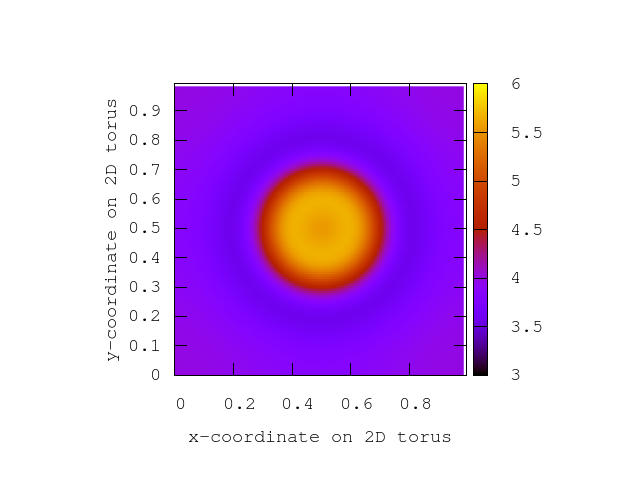}
\figcaption{PDE solution, $t=0.01$}
\label{fig:PDE3}
\end{minipage}
\\[\intextsep]
\begin{minipage}{0.5\linewidth}
\centering
\includegraphics[width=\linewidth]{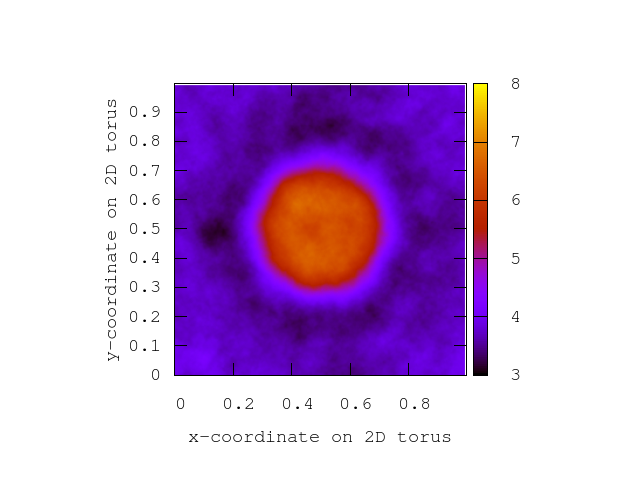}
\figcaption{Cell density for particle system, $t=0.04$}
\label{fig:ce4}
\end{minipage}
\begin{minipage}{0.5\linewidth}
\centering
\includegraphics[width=\linewidth]{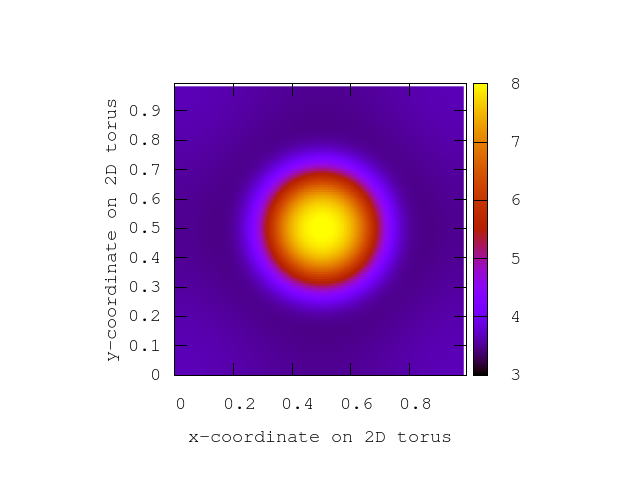}
\figcaption{PDE solution, $t=0.04$}
\label{fig:PDE4}
\end{minipage}
\\[\intextsep]
\begin{minipage}{0.5\linewidth}
\centering
\includegraphics[width=\linewidth]{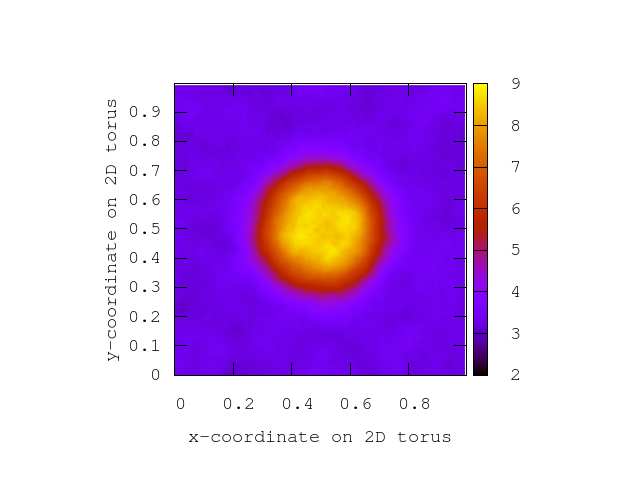}
\figcaption{Cell density for particle system, $t=0.2$}
\label{fig:ce5}
\end{minipage}
\begin{minipage}{0.5\linewidth}
\centering
\includegraphics[width=\linewidth]{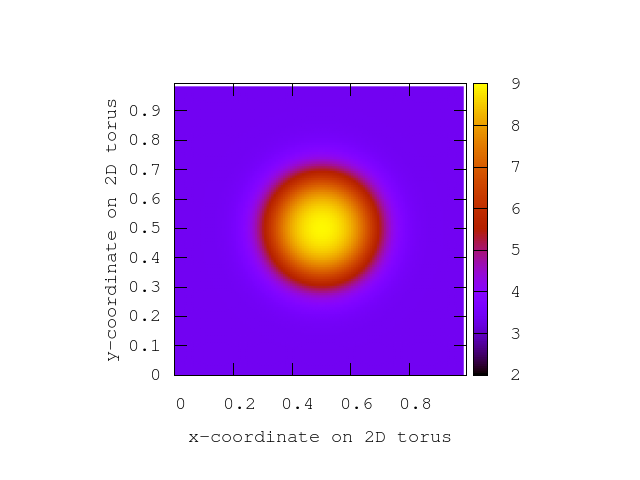}
\figcaption{PDE solution, $t=0.2$}
\label{fig:PDE5}
\end{minipage}
\\[\intextsep]
The simulations for the cell density exhibit an interesting phenomenon near the boundary of the 
support of the positive density of the chemo-attractant, which itself is not  
visualized in the figures. 
Initially a
ring of lower cell density forms around the area where the cells finally will concentrate 
due to positive chemotaxis, best visible as the blue ring in Figures 2 and 4. This can be explained by
particles jumping out of the region of this ring into the domain of positive chemical attraction where they get stuck 
due to their lower motility there. If they jump out of this domain again, then they
may not move very far away from it by random motion. This leads to a density profile which is a non-monotone function of the distance from the centre. For longer times this effect vanishes, however, and the density approaches the steady
state, which (in the hydrodynamic limit as $N\to\infty$) is numerically indistinguishable from the solution of the PDE
at time $t=0.2$ shown in Figure \ref{fig:PDE5}. 
Initially, the fluctuations of the initial condition dominate in the simulation of the particle system, which explains the rather large difference to the deterministic solution of the PDE. Averaging over several realizations of the particle system would lead to closer resemblance between the two.
%
%
\\

Finally, let us remark that it is possible to also consider models with 
a strong localization of particles. 
Again we assume that the chemical environment results from a random
process. 
For the case considered below we do not have a proof for a hydrodynamic limit so far,
but expect a similar outcome as before.
Also, we do not have stationarity in this case but convergence towards
a stationary random variable.
Now we simulate only the respective particle model, since
the solution for the limiting PDE would require more refined
numerical schemes than given above, because now blow-up phenomena may occur. \\
Our example here is the following, let us
consider a zero range process within the range of so-called condensation. 
Specifically, we consider a ZRP
with a constant jump rate $g(n) = g_0$ for all $n > 0$. As soon as $\varphi > g_0$, the partition function $Z$ defined
in \eqref{Z} does not converge and the corresponding grand-canonical measure (\ref{GC}) does not exist. As shown in \cite{krug,landim},
for large enough densities the particles tend to concentrate on only a few sites with the lowest exit rate. 
%
%
However, in this case, this behavior is not due to the
reinforced interaction of the cells and the attractive chemical molecules, 
but instead it is solely due to the stochastic properties of the
ZRP. 
Such strong localization effects may relate to finite time blow-up of 
solutions of related limiting objects.
\\

Figs.\ \ref{fig:cond1}--\ref{fig:cond5} show the results of simulations for
this model with $N=100$, $g(n) = 3$, and all other
parameters given as before, i.e.\ uniform initial configuration 
$\eta(x) = 4$, $\chi_0 = 2$, $\nu = 0.5$ and $\vartheta$ given
in \eqref{CA_dens} with final averaging over the empirical measure with radius $0.05$. In fact, this radius can be seen
explicitly at points where many particles have clustered in the figures below.
Figure \ref{fig:cond_CA} shows the random density of the chemo-attractant in
order to
compare with the location of the cell aggregate. 
\\[\intextsep]
\begin{minipage}{0.5\linewidth}
\centering
\includegraphics[width=\linewidth]{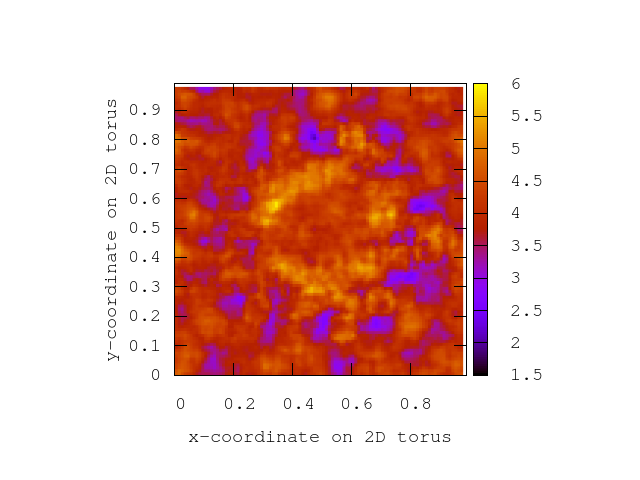}
\figcaption{Cell density, $t=0.008$}
\label{fig:cond1}
\end{minipage}
\begin{minipage}{0.5\linewidth}
\centering
\includegraphics[width=\linewidth]{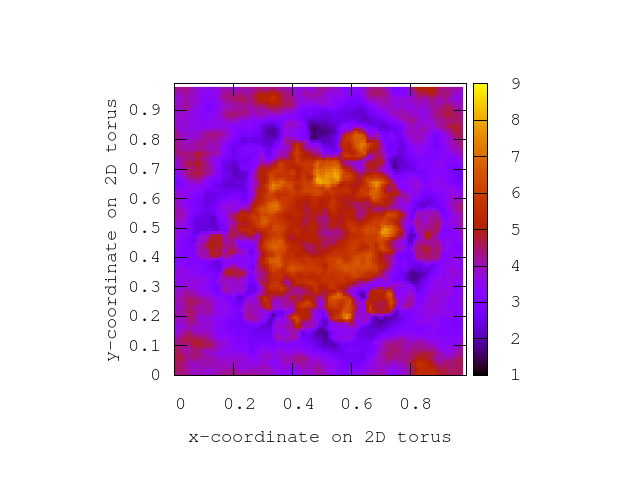}
\figcaption{Cell density, $t=0.04$}
\label{fig:cond2}
\end{minipage}
\\[\intextsep]
\begin{minipage}{0.5\linewidth}
\centering
\includegraphics[width=\linewidth]{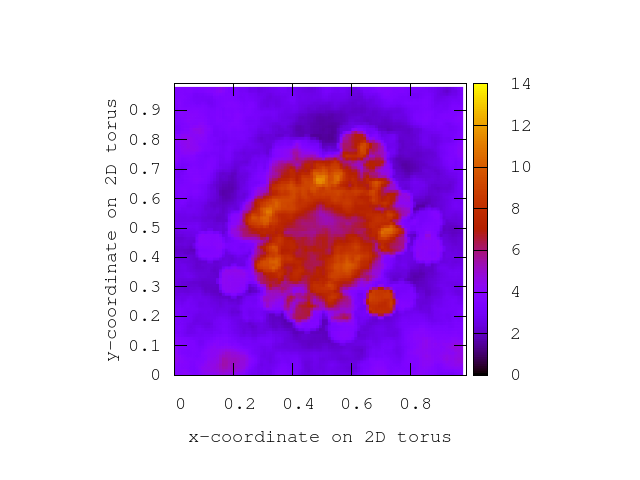}
\figcaption{Cell density, $t=0.1$}
\label{fig:cond3}
\end{minipage}
\begin{minipage}{0.5\linewidth}
\centering
\includegraphics[width=\linewidth]{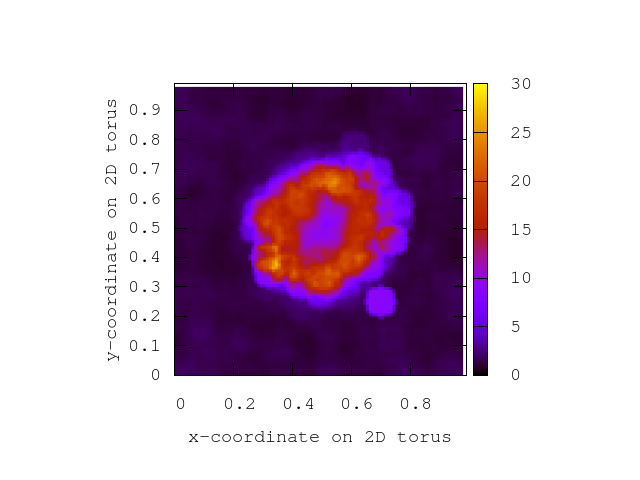}
\figcaption{Cell density, $t=0.4$}
\label{fig:cond4}
\end{minipage}
\\[\intextsep]
\begin{minipage}{0.5\linewidth}
\centering
\includegraphics[width=\linewidth]{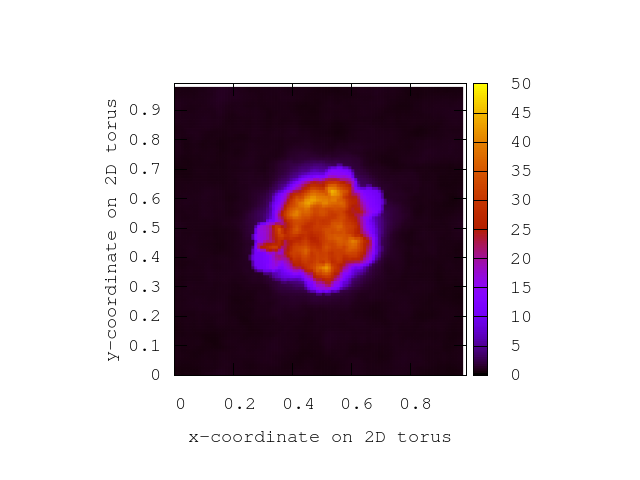}
\figcaption{Cell density, $t=2$}
\label{fig:cond5}
\end{minipage}
\begin{minipage}{0.5\linewidth}
\centering
\includegraphics[width=\linewidth]{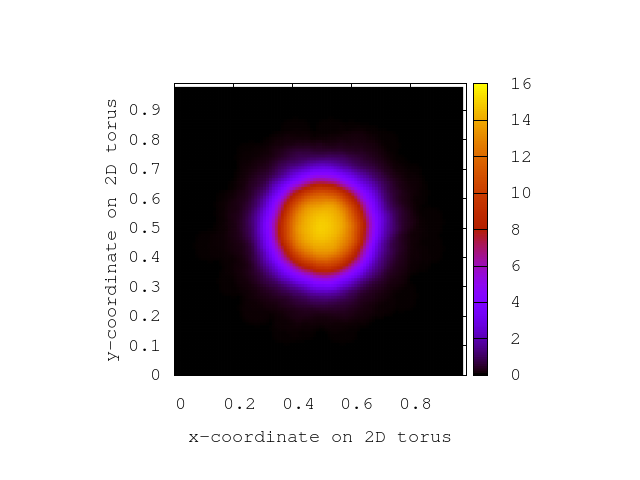}
\figcaption{ Density of the chemo-attractant $\vartheta$ (random)}
\label{fig:cond_CA}
\end{minipage}
\\[\intextsep]
If the total density in the system is high enough, we expect that eventually a finite fraction of all particles will concentrate on the site with the highest number of chemical molecules.

\section{Concluding remarks}
In this paper we derived the first equation of a general class of
chemotaxis systems via a hydrodynamic limit of a stochastic lattice gas.
The attractive chemical environment is prescribed in our case and assumed
to be random and stationary, with a slowly varying mean.
The situation of very strong clustering of the particles on single lattice sites
is excluded in our theory, although this phenomenon can happen during
selforganization of chemotactic particles in certain situations, and then it is of  important biological relevance, like for the slime mold
amoebae {\sl Dictyostelium discoideum}. It would be interesting to be
able to derive the full chemotaxis system from a hydrodynamic limit.
Technically this is more challenging and methods of scale separation
might play a role here.\\

\centerline{\bf{Acknowledgements}}
Major parts of this work were done while D.M. and A.S. were working
at the University of Heidelberg. This work was finished while A.S. took part
in the CGP-program, 2015 at the Isaac Newton Institute, Cambridge.


\end{document}